\documentclass[12pt]{amsart}
\usepackage{amsmath,mathtools}
\usepackage{amssymb}
\usepackage{srcltx}
\usepackage{xcolor}
\usepackage{xspace}
\usepackage{enumerate}

\usepackage{tikz-cd}
\usetikzlibrary{arrows.meta}

\newcommand{\res}{\mathord{\upharpoonright}}
\newcommand{\Fraisse}{Fra\"iss\'e\xspace}

\newtheorem{theorem}{Theorem}
\newtheorem{lemma}[theorem]{Lemma}
\newtheorem{corollary}[theorem]{Corollary}
\newtheorem{proposition}[theorem]{Proposition}
\newtheorem*{claim}{Claim}

\theoremstyle{definition}
\newtheorem{definition}[theorem]{Definition}
\newtheorem{remark}[theorem]{Remark}

\newtheorem{example}[theorem]{Example}

\begin{document}
\title{Higher dimensional obstructions for star reductions}

\author{Alex Kruckman}
\address{Department of Mathematics and Computer Science\\Wesleyan University\\Science Tower 655\\265 Church Street\\Middletown, CT  06459-0128\\USA}
\email{akruckman@wesleyan.edu}
\urladdr{https://akruckman.faculty.wesleyan.edu/}

\author{Aristotelis Panagiotopoulos}
\address{Mathematics Department, Caltech, 1200 E. California Blvd,
Pasadena, CA 91125}
\email{panagio@caltech.edu}
\urladdr{http://www.its.caltech.edu/~panagio/}

\subjclass[2000]{Primary 03E15; Secondary  54H05}

\thanks{This work greatly benefited from a visit of Alex Kruckman at the California
Institute of Technology in the Spring 2018. 
The authors gratefully acknowledge the hospitality and the financial support of the Institute.}

\keywords{Polish group, Polish
space, Borel reduction, Baire measurable, $*$-reduction, category preserving, isomorphism, n-amalgamation}

\begin{abstract}
A $*$-reduction between two equivalence relations is a Baire measurable reduction which preserves generic notions, i.e., preimages of meager sets are meager. We show that a $*$-reduction between orbit equivalence relations induces generically an embedding between the associated Becker graphs. We introduce a notion of dimension for Polish $G$-spaces which is generically preserved under $*$-reductions. For every natural number $n$ we define a free action of  $S_{\infty}$ whose dimension is $n$ on every invariant Baire measurable non-meager set. We also show that the $S_{\infty}$-space which induces the equivalence relation  $=^{+}$  of countable sets of reals is
$\infty$-dimensional on every invariant Baire measurable non-meager set. We conclude that the orbit equivalence relations associated to all these actions are pairwise incomparable with respect to $*$-reductions.
\end{abstract}

\maketitle

\section{Introduction}

Many classification problems in mathematics can be formally presented as pairs $(X,E)$, where $X$ is a
Polish space and $E$ is an analytic equivalence relation on $X$. For example, the problem of classifying graph structures on domain $\omega$ up to isomorphism, or the problem of classifying self-adjoint operators on a separable Hilbert space up to unitary equivalence, are both instances of this formal setup. In order to compare the relative complexity of two such problems $(X,E),(Y,F),$  one often wants to know whether there exists a map $f\colon X\to Y$ such that: 
\begin{enumerate}
\item $f$ is a {\bf reduction} from $E$ to $F$, i.e., $x E x'\iff f(x)F f(x')$, for all $x,x'\in X$; 
\item $f$ preserves some structural properties of $(X,E)$ and $(Y,F)$.
\end{enumerate}
In practice, and since otherwise the question trivializes,\footnote{Assuming the axiom of choice, there is an ``abstract'' reduction from $E$ to $F$ if and only if $|X/E|\leq |Y/F|$.}  the reduction $f$ is usually assumed to be Borel, or at least Baire measurable. Besides this minimal definability requirement, one often wants $f$ to be sensitive to various other structural properties of $(X,E)$ and $(Y,F)$. For example, in the context of the topological Vaught conjecture it is useful to consider \emph{faithful reductions}, i.e., Borel reductions $f$ from $E$ to $F$ with the additional property that the $F$-saturation of the image of any $E$-invariant Borel subset of $X$ is a Borel subset of $Y$ \cite{Friedman, Hjorth, Gao2, Gao3}. In the context of ergodic theory, when $X,Y$ additionally support probability measures $\mu,\nu$, and $E,F$ are orbit equivalence relations of measure preserving actions of countable groups, one often works with \emph{orbit equivalences}. An orbit equivalence is a Borel bijection $f\colon X_0\to Y_0$ between Borel conull subsets  $X_0\subseteq X$ and $Y_0\subseteq  Y$, which is a reduction from $E\res X^{2}_0$ to $F\res Y^{2}_0$ and which  pulls back $\nu$-null sets to $\mu$-null sets \cite{Gab,GabF,KechrisMiller}. 
In this paper we introduce and study \emph{$*$-reductions}. These are  
Baire measurable reductions which are category preserving.\footnote{For more information on category preserving maps between Polish spaces one may consult  \cite[Appendix A]{Category}.}

\begin{definition}
Let $X$ and $Y$ be Polish spaces. A map $f\colon X\to Y$ is {\bf category preserving} if for every meager subset $M$ of $Y$,  $f^{-1}(M)$ is a meager subset of $X$. A {\bf $*$-map} is a Baire measurable category preserving map. 
\end{definition}

Recall that a map $f\colon X\to Y$ is Baire measurable if the preimage of every open set is Baire measurable. If $f$ is a $*$-map, it satisfies the stronger property that the preimage of every Baire measurable set is Baire measurable. Indeed, a Baire measurable set is the symmetric difference of an open set and a meager set, so its preimage is the symmetric difference of a Baire measurable set and a meager set, which is again Baire measurable. As a consequence, the composition of two $*$-maps is a $*$-map. 

\begin{definition}
Let $E$ and $F$ be equivalence relations on Polish spaces $X$ and $Y$, respectively.
A {\bf $*$-reduction} from $E$ to $F$ is a $*$-map $f\colon X\to Y$ which is a reduction from $E$ to $F$. We write $E\leq_* F$ when such a $*$-reduction exists.
\end{definition}

Since the composition of two $*$-reductions is a $*$-reduction, the relation $\leq_*$  induces a preordering among classification problems $(X,E)$. This  preordering reflects the relative complexity between two such problems from the point of view of $*$-reductions. Showing that some classification problem $*$-reduces to another often just amounts to finding a ``canonical way'' of coding $E$ into $F$. However, showing negative results predicates upon developing a basic obstruction theory for $*$-reductions.
When it comes to simple Borel and Baire measurable reductions, there are many well known descriptive set theoretic and dynamical obstructions \cite{Hjorth, Gao, GOP, Friedman}.  Similar obstructions have been developed for orbit equivalences. For example, Gaboriau's theory of cost  \cite{Gab,GabF} implies that two free groups of different rank can never produce orbit equivalent equivalence relations via free and measure preserving actions on a standard measure space. The main goal of this paper is to develop certain obstructions for $*$-reductions by advancing further some of the techniques introduced in \cite{GOP}.

We briefly describe here the main ideas; definitions and details can be found in Section \ref{S: Becker graph embeddings}. Given a Polish $G$-space $X$, let $E^G_X$ be the associated orbit equivalence relation on $X$. In \cite{GOP}, the authors introduce a digraph structure $\mathcal{B}(X/G)$ on the quotient $X/E^G_X$ based on Becker's notion of right-embeddings.\footnote{In \cite{Becker}, Becker studies left-embeddings. Here and in \cite{GOP}, with an eye on the applications, it is convenient to develop everything in terms of right-embeddings. Analogous results hold for left-embeddings.} 
The \emph{right-Becker graph} $\mathcal{B}(X/G)$ has the elements $[x]$ of $X/E^G_X$ as vertices, and an arrow $[x]\to[y]$ whenever there is a right-Cauchy sequence $(g_n)$ in $G$ so that $g_n y$ converges to $x$. The main structural result in \cite{GOP} (\cite[Proposition 2.8]{GOP}) states that a Baire measurable reduction $f$ from $E^G_X$ to $E^H_Y$ induces an injective digraph homomorphism from $\mathcal{B}(X_0/G)$ to $\mathcal{B}(Y/H)$, where $X_0$ is an invariant generic subset of $X$.  Using this result and the fact that the Becker graph associated to the action of a CLI group $H$ is trivial,\footnote{A group is CLI if it admits a complete and left-invariant metric.} they introduce---in analogy to Hjorth's turbulence condition---a new dynamical obstruction for classifying $E^G_X$ by CLI-group actions. Here we show that for $*$-reductions, the main structural result from \cite{GOP} can be strengthened.

\begin{theorem}\label{T: main}
Suppose that $G,H$ are Polish groups, $X$ is a Polish $G$-space, $Y$ is a Polish $H$-space, and $f\colon X\to Y$ is $*$-reduction from $E^G_X$ to $E^H_Y$. Then there is a $G$-invariant dense $G_{\delta}$ subset $X_0$ of $X$ and an $H$-invariant non-meager Baire measurable subset $Y_0$ of $Y$ so that the induced map $[f]$  restricted to $X_0/G$ is an isomorphism from the digraph $\mathcal{B}(X_0/G)$ to the digraph $\mathcal{B}(Y_0/H)$.
\end{theorem}
This strengthening allows us to utilize ``higher dimensional'' properties of the Becker graphs $\mathcal{B}(X/G)$ and $\mathcal{B}(Y/H)$ 
as obstructions for $*$-reducing $E^G_X$ to $E^H_Y$. We say that the \emph{dimension} of the Becker graph $\mathcal{B}(X/G)$ is at least $n$, if the combinatorial $n$-cube embeds in $\mathcal{B}(X/G)$. We say that the \emph{generic dimension} of $\mathcal{B}(X/G)$ is at least $n$ if the combinatorial $n$-cube embeds in $\mathcal{B}(X_0/G)$,  for every invariant comeager subset $X_0$ of $X$. Similarly, we say the \emph{locally generic dimension of $\mathcal{B}(X/G)$} is at least $n$ if the combinatorial $n$-cube embeds in $\mathcal{B}(X_0/G)$, for every invariant non-meager subset $X_0$ of $X$ with the Baire property.

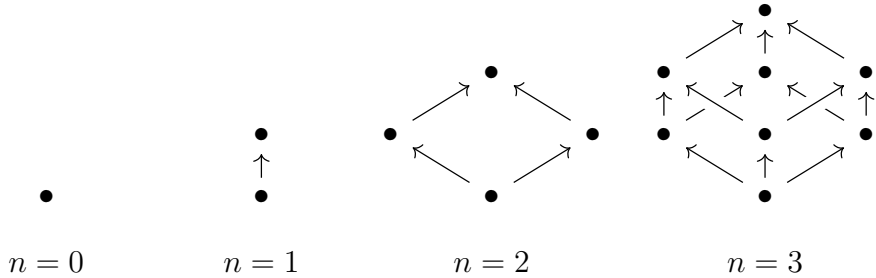
\begin{figure}[!ht]
\centering
\begin{tikzcd}[column sep=0.9em,row sep=0.8em]
& & & & & & &  & & & \bullet & \\
& & & & & & & \bullet &  & \bullet \arrow[ur]  & \bullet \arrow[u] \arrow[from=dr]& \bullet \arrow[ul] \\ 
& & & & \bullet & & \bullet\arrow[ur] & & \bullet \arrow[ul] & \bullet\arrow[u]\arrow[ur] & \bullet  \arrow[ul, crossing over] \arrow[ur, crossing over]& \bullet \arrow[u]\\
\bullet & & & & \bullet\arrow[u] & & & \bullet \arrow[ul]\arrow[ur] & & & \bullet \arrow[ul]\arrow[u]\arrow[ur] & \\
n=0 & & & & n=1 & & & n=2 & & & n=3 & \\
\end{tikzcd}
\caption{The combinatorial $n$-cube.}\label{Picture1}
\end{figure}

From this point of view, the anti-classification criterion in \cite{GOP} corresponds to the fact that CLI groups induce $0$-dimensional orbit equivalence relations and hence they cannot reduce --- in a Baire measurable fashion --- to orbit equivalence relations which are generically at least $1$-dimensional. In Section \ref{S: Applications}, we define for every $n>0$ a free action  of $S_{\infty}$ on a Polish space  of countable structures $\mathrm{Mod}_\omega(\widehat{\mathcal{B}_{n}})$  which, as we show in Theorem \ref{thm:bklmain}, is locally generically $(n-1)$-dimensional but not  $n$-dimensional. Theorem~\ref{T: main} then implies that the associated classification problems $(\mathrm{Mod}_\omega(\widehat{\mathcal{B}_{n}}),\simeq_{\mathrm{iso}})$, for $n>0$, are incomparable under $*$-reductions. The structures in  $\mathrm{Mod}_\omega(\widehat{\mathcal{B}_{n}})$ are \emph{labeled} versions of certain families of structures that were introduced and studied in~\cite{BKL} for their interesting behavior with respect to disjoint $n$-amalgamation.

In the process of proving Theorem \ref{thm:bklmain} we develop a general method which can be used for computing the generic dimension of other similar problems; see Remark \ref{Remark fin}. In particular, our method implies that $=^{+}$ is the orbit equivalence relation of an $\infty$-dimensional $S_{\infty}$-space and is therefore incomparable under $*$-reductions to
$(\mathrm{Mod}_\omega(\widehat{\mathcal{B}_{n}}),\simeq_{\mathrm{iso}})$ for all $n>0$.

\subsubsection*{Acknowledgments}
We are grateful to Alexander Kechris, Martino Lupini, and Ronnie Chen for many helpful discussions and suggestions.

\section{$*$-reductions and Becker graphs}\label{S: Becker graph embeddings}

In what follows, $X$ and $Y$ will be Polish spaces and $E$ and $F$ will be analytic equivalence relations on $X$ and $Y$, respectively. Often these equivalence relations will be  \emph{orbit equivalence relations} associated  with a Polish group action.  
Let $G$ be a Polish group. A {\bf Polish $G$-space} is a Polish space $X$ together with a continuous action $G\times X\to X$ on $X$. The associated {\bf orbit equivalence relation} $E^G_X$ on $X$ is defined by $x E^G_X y \iff [x]=[y]$, where $[x]$ denotes the orbit of $x\in X$ under the action. If $A$ is any subset of $X$, we will denote by $[A]$ the $G$-saturation of $A$: $\{x\mid \exists g\in G\, g x\in A\}$. Note that if $A$ is an analytic subset of $X$, then $[A]$ is analytic. 

Some differences are immediately apparent between the usual definable reductions and $*$-reductions. For example, since the preimage of a comeager set under a category preserving map is comeager, an equivalence relation $E$ cannot $*$-reduce to an equivalence relation $F$ with a comeager $F$-class, unless $E$ also has a comeager $E$-class.  As a consequence, and in contrast to the Borel reduction hierarchy \cite{Friedman}, countable graph isomorphism is not $\leq_*$-universal amongst orbit equivalence relations of $S_{\infty}$ actions. 
On the other hand, if $E$ and $F$ each have a comeager equivalence class, then $E\leq_* F$ if and only if $|X/E|\leq |Y/F|$ and there is some $*$-map from the comeager $E$-class to the comeager $F$-class.  The problem becomes more interesting when both $E$ and $F$ have only meager equivalence classes, and we will typically be interested in this case.

 In general, $*$-reductions reveal to be much more sensitive to the dynamical aspects of the classification problems under comparison compared to the usual definable reductions. For example, we have the following proposition. Recall that an equivalence relation $E$ on $X$ is {\bf generically ergodic} if every Baire measurable $E$-invariant subset of $X$ is either meager or comeager.

\begin{proposition}\label{P: gen erg}
Let $E$ and $F$ be analytic equivalence relations on Polish spaces $X$ and $Y$ respectively. If $E\leq_* F$ and $F$ is generically ergodic, then $E$ is generically ergodic. 
\end{proposition}
\begin{proof}
Let $f$ be the $*$-reduction from $E$ to $F$ and let $A$ be a Baire measurable $E$-invariant subset of $X$. Let $B$ be a $G_{\delta}$ subset of $A$ so that $A\setminus B$ is meager. Every Baire measurable map is continuous when restrict to some comeager set, so by restricting $B$ further if necessary we can assume that $f\res B$ is continuous. Since $f\res B$ is continuous, its image $f(B)$ is analytic, and hence so is its $F$-saturation $C = [f(B)]$. Then $C$ is an $F$-invariant Baire measurable subset of $Y$, and it is therefore either meager or comeager. Since $f$ is category preserving and $A\setminus M \subseteq f^{-1}(C)\subseteq A$ for some meager set $M$, it follows that $A$ is meager or comeager as well. 
\end{proof}

We restrict now our attention entirely to orbit equivalence relations. Recall that a sequence $(g_n)$ in $G$ is Cauchy with respect to some left invariant metric on $G$ if and only if it is Cauchy with respect to any left invariant metric on $G$; see \cite{Becker}. In this case we say that $(g_n)$ is {\bf left-Cauchy}. Similarly we define when $(g_n)$ is right-Cauchy. Let $X$ be a Polish $G$-space and let $x,x'\in X$. Following Becker \cite{Becker}, we say that {\bf $x$  left-embeds in $x'$} if there is a left-Cauchy sequence $(g_n)$ so that $(g_n x)$ converges to $x'$. Similarly we say that {\bf $x$  right-embeds in $x'$} if there is a right-Cauchy sequence $(g_n)$ so that $(g_n x')$ converges to $x$. 

We recall now some definitions and a result from \cite{GOP}. In view of our applications in Section \ref{S: Applications}, we are going to develop here everything in terms of right-embeddings. 
All results hold equally for left-embeddings. The {\bf right Becker digraph} $\mathcal{B}(X/G)$ associated to a Polish $G$-space is a graph on domain $X/G=\{[x]\mid x\in X\}$, whose arrows are precisely all pairs $([x],[x'])$ such that $x$ right-embeds in $x'$.  The main anti-classification result developed in \cite{GOP} was a consequence of the following proposition.  
Recall that an {\bf $(E,F)$-homomorphism} is any map $f\colon X\to Y$ with \[x E x' \implies f(x) F f(x').\]  An $(E,F)$-homomorphism $f$ induces a map $[f]\colon X/E\to Y/F$ between the quotients, sending $[x]$ to $[f(x)]$. This map is injective if and only if $f$ is a reduction. 

\begin{proposition}[\cite{GOP}, Proposition 2.8]\label{Prop}
Let $G,H$ be Polish groups, let $X$ be a Polish $G$-space, and let let $Y$ be a Polish $H$-space. If $f\colon X\to Y$ is Baire measurable $(E^G_X,E^H_Y)$-homomorphism, then there is a $G$-invariant dense $G_{\delta}$ subset $X_0$ of $X$ so that the induced map $[f]\colon X_0/G\to Y/H$ is a graph homomorphism from the digraph $\mathcal{B}(X_0/G)$ to the digraph $\mathcal{B}(Y/H)$.
\end{proposition}

The above proposition was used in \cite{GOP} in order to show that certain equivalence relations do not reduce in a Baire measurable fashion to any orbit equivalence relation of a CLI group action. In particular, if $H$ is CLI then $\mathcal{B}(Y/H)$ contains only loops~\cite[Lemma 2.7]{GOP}. Hence if $\mathcal{B}(X/G)$ contains non-trivial edges in every invariant $G_\delta$ subset $X_0$ of $X$ then, by Proposition  \ref{Prop}, every Baire measurable $(E^G_X,E^H_Y)$-homomorphism would fail to be a reduction.

In the context of $*$-reductions, we can strengthen the conclusion of Proposition \ref{Prop} so that $[f]\res (X_0/G)$ is an embedding of digraphs rather than an injective digraph homomorphism. This is the essence of Theorem \ref{T: main}. For the proof we will need the following  minor strengthening of \cite[Lemma 2.5]{GOP}\footnote{This is essentially \cite[Lemma 3.17]{Hjorth} modified as in the beginning of the proof of \cite[Theorem 3.18]{Hjorth}.}.

\begin{lemma}\label{L: orbit continuity}
Suppose that $G,H$ are Polish groups, $X$ is a Polish $G$-space, and $Y$ is a Polish $H$-space. Let $C\subseteq X$ be a $G_{\delta}$ subset of $X$ such that for any $x\in C$ the set $\{g\in G \mid gx\in C \}$ is comeager in $G$. Let $f\colon C\to Y$ be a Baire measurable homomorphism from the equivalence relation $E^G_X\res C$ to $E^H_Y$. Then there is a dense in $C$, $G_{\delta}$ subset $\widetilde{C}$ of $C$ so that: 
\begin{enumerate}[(a)]
\item $f\res \widetilde{C}$ is continuous;
\item for any $x\in \widetilde{C}$ the set $\{g\in G \mid gx\in \widetilde{C} \}$ is comeager in $G$;
\item for any $x_0\in \widetilde{C}$ and for any open neighborhood $W$ of the identity in $H$ there exists an open neighborhood $U$ of $x_0$ and an open neighborhood $V$ of the identity of $G$ such that for any $x\in U\cap \widetilde{C}$ and for a comeager set of $g\in V$, we have that $f(gx)\in W f(x)$ and $gx\in \widetilde{C}$.
\end{enumerate}
\end{lemma}
\begin{proof}

Let $V\subseteq_1 G$ stand for ``$V$ is an open neighborhood of identity in $G$'' and let $\forall^*x\in A$ stand for ``for a comeager collection of elements in $A$.''  

The proof is exactly the same as in \cite[Lemma 2.5]{GOP} so we will omit the details. First one needs to show that for any fixed $W\subseteq_1 H$ we have that 
\[\forall x_0\in C \;\; \forall^* g_0\in G \;\; \; \exists V\subseteq_1 G \;\; \; \forall^{*} g_1 \in V \;\;\; f(g_1 g_0 x_0)\in  W f(g_0 x_0).\]
Then by an application of Kuratowski-Ulam theorem, and since by assumption we have that $\forall x\in C \; \;\forall^* g\in G \; \;   gx\in C$, we get a dense $G_{\delta}$ subset $C_0$ of $C$ so that 
\[\forall x\in C_0 \;\; \; \exists V\subseteq_1 G \;\; \; \forall^{*} g \in V \;\;\; f(g x)\in  W f(x).\]
Then we use $C_0$ exactly as in the proof of \cite[Lemma 2.5]{GOP} to define $C_1$ and we finish by setting $\widetilde{C}:=\{x\in C_1 \mid \forall^* g\in G\; \; gx\in C_1\}$. 
\end{proof}

We now proceed to the proof of Theorem \ref{T: main}.

\begin{proof}[Proof of Theorem \ref{T: main}]
Let $f\colon X\to Y$ be as in the statement of the theorem.  
As in \cite[Lemma 2.8]{GOP}, if $\widetilde{C}$ is the set provided by Lemma \ref{L: orbit continuity} with $C=X$, then setting $X'_0=\{x\in X\mid \forall^* g\in G \; gx\in \widetilde{C}\}$ we get the invariant dense $G_{\delta}$ set appearing in the statement of Proposition \ref{Prop}, i.e. we get that $[f]\colon X'_0/G\to Y/H$ is a homomorphism from $\mathcal{B}(X'_0/G)$ to $\mathcal{B}(Y/H)$. For the convenience of the reader we include a brief sketch: assume that $x,x'$ are in $X'_0$  and $x$ right-embeds in $x'$. Since this is a $G$-invariant property we can assume that $x,x'$ are in $\widetilde{C}$. Let $(g_n)$ be a right-Cauchy sequence with $g_n x'\to x$. As in \cite[Lemma 2.3]{GOP} we can use properties (b) and (c) of Lemma \ref{L: orbit continuity} to slightly modify $(g_n)$ so that $g_nx$ is in $\widetilde{C}$ and so that there is a right-Cauchy sequence $(h_n)$ in $H$ with $f(g_n x)=h_n f(x)$, for all $n>0$. By property (a) of Lemma \ref{L: orbit continuity} we get that $h_n f(x)\to f(x')$.

Since $f$ is by assumption a reduction we also have that $[f]$ is injective.
In what follows, we will intersect $X'_0$ with another invariant dense $G_{\delta}$ set $X''_0$ with the property that for every $x,x'\in X''_0$, if $f(x)$ right-embeds in $f(x')$, then $x$ right-embeds in $x'$. The desired sets will then be: $X_0=X'_0\cap X''_0$, and $Y_0=[f(X_0)]$.

\begin{claim}
There exists $D\subseteq Y$ and $h\colon D\to X$ so that: 
\begin{enumerate}
\item $D$ is a $G_{\delta}$ subset of $Y$ with $f^{-1}(D)$ dense $G_{\delta}$ in $X$;
\item for every  $ y\in D$  we have that $\{g\in H \mid g y \in D\}$ is a comeager subset of $H$;
\item $h$ is a Baire measurable reduction from $E^H_Y\res D$ to $E^G_X$, with $[f]\circ [h]=\mathrm{id}$ and $[h]\circ [f\res f^{-1}(D)]=\mathrm{id}$.
\end{enumerate}
\end{claim}
\begin{proof}[Proof of Claim.]
Since $f$ is Baire measurable, we can find a dense $G_{\delta}$ subset $C_0$ of $X$ so that $f\res C_0$ is continuous on $C_0$. Let $A=[f(C_0)]\subseteq Y$ be the $H$-saturation of the image of $C_0$ under $f$. 
Since $F=\{(y,x)\in A\times C_0 \mid y E^H_Y f(x)\}$ is an analytic set, by the Jankov--von Neumann selection theorem \cite[Theorem 18.1]{Kechris} there is a $\sigma(\Sigma^1_1)$-measurable, and hence Baire measurable, map $h'\colon A \to C_0$ uniformizing $F$. Let $D_0$ be a $G_\delta$ subset of $A$ so that $B:=A\setminus D_0$ is meager. We claim that the set $D=\{y\in D_0 \mid \forall^* g\in H\, g y \in D_0\}$ and the map $h=h'\res D$ are as required. To see this notice that property (3) follows immediately from the fact that $h'$ uniformizes $F$, and property (2) follows from \cite[Proposition 3.2.5(v)]{Gao}. For property (1), $D$ is $G_{\delta}$ since it is the Vaught transform of a $G_{\delta}$ set intersected with a $G_{\delta}$ set; see \cite[Proposition 3.2.7]{Gao}. Finally, notice that since $A$ is $H$-invariant we have that 
\[f^{-1}(D)= f^{-1}\bigg(\bigg\{y\in (Y\setminus B)  \mid \forall^* g\in H\, g y \in (Y\setminus B)\bigg\}\bigg),\]
and the set $\big\{y\in (Y\setminus B)  \mid \forall^* g\in H\, g y \in (Y\setminus B)\big\}$ is comeager in $Y$ by Kuratowski-Ulam. Since $f$ preserves category and it is continuous on the dense $G_{\delta}$ subset $C_0$ of $X$,  $f^{-1}(D)$ is a dense $G_{\delta}$ subset of $X$ as well.
\end{proof}
By Lemma \ref{L: orbit continuity} applied to $h\colon D\to X$ we can now find a dense $G_\delta$ subset $\widetilde{D}$ of $D$ satisfying the conclusions of Lemma \ref{L: orbit continuity}. 
By the proof of \cite[Proposition 2.8]{GOP}, as described in the first paragraph of this proof, for every $y,y'\in \widetilde{D}$, if $y$ right-embeds in $y'$ then $h(y)$ right-embeds in $h(y')$.

Setting  now $X''_0$ to be the Vaught transform $\{x\in X \mid \forall^*g\in G \; gx\in f^{-1}(\widetilde{D})\}$ of $f^{-1}(\widetilde{D})$  we get the desired sets: $X_0=X'_0\cap X''_0$ and $Y_0=[f(X_0)]$.
To see this, let $x,x'\in X''_0$ so that  $f(x)$ right-embeds in $f(x')$. By replacing $x,x'$ with a generic translate we can assume without the loss of generality that $x,x'\in f^{-1}(\widetilde{D})$. Since $f(x),f(x')\in \widetilde{D}$ we get that $h(f(x))$ right-embeds in $h(f(x'))$. Property (3) of the above claim implies that $x$ right-embeds in $x'$. Finally, since $[f(X_0)]=[f(X_0\cap\widetilde{C})]$ and $f\res\widetilde{C}$ is continuous, we have that $Y_0$ is analytic and hence Baire measurable. Moreover, since $f^{-1}(Y_0)=X_0$,  $f$ is category preserving, and $X_0$ is comeager in $X$, $Y_0$ cannot be meager in $Y$.  
\end{proof}

\section{Higher dimensional obstructions}\label{S: Applications}

For every $n\geq 0$, the {\bf combinatorial $n$-cube} is the poset category $\Delta^{n-1}$, whose set of objects is the powerset $\mathcal{P}(\{0,\ldots,n-1\})$ of $[n] = \{0,\ldots,n-1\}$ and whose arrows are precisely the inclusions $\sigma\subseteq\tau$ between subsets $\sigma,\tau$ of $\{0,\ldots,n-1\}$. For $n=0$, the combinatorial $n$-cube is the one-object poset $\Delta^{-1}=\mathcal{P}(\emptyset)=\{\emptyset\}$.
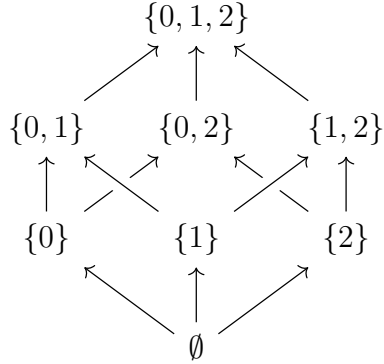
\begin{figure}[!ht]
\centering
\begin{tikzcd}[column sep=1 em,row sep=0.9em]
& \{0,1,2\} & \\
& & \\
\{0,1\} \arrow[uur]  & \{0,2\} \arrow[uu] \arrow[from=ddr]& \{1,2\} \arrow[uul] \\ 
& & \\
\{0\}\arrow[uu]\arrow[uur] & \{1\}  \arrow[uul, crossing over] \arrow[uur, crossing over]& \{2\} \arrow[uu]\\
& & \\
& \emptyset \arrow[uul]\arrow[uu]\arrow[uur] & \\
\end{tikzcd}
\caption{Hasse diagram of the combinatorial $3$-cube.}
\end{figure}

Forgetting for a moment the composition law between the arrows we can view $\Delta^{n-1}$ as a digraph. Let $X$ be a Polish $G$-space and let $C$ be a non-empty $G$-invariant subset of $X$. The {\bf dimension of $C$} is the largest natural number $n\geq 0$ so that the combinatorial $n$-cube embeds in $\mathcal{B}(C/G)$, if such $n$ exists; and it is $\infty$ otherwise. Notice that for every $C$ as above the dimension of $C$ is at least $0$. We say that $X$ is {\bf generically $n$-dimensional} if $n$ is the largest element in  $\{0,1,\ldots\}\cup\{\infty\}$ so that the dimension of every invariant comeager subset $X$ is at least $n$. Similarly we say that $X$ is {\bf locally generically $n$-dimensional} if $n$ is the largest element in  $\{0,1,\ldots\}\cup\{\infty\}$ so that the dimension of every invariant non-meager subset of $X$ with the Baire property is at least $n$. If $\mathrm{dim}(X)$ denotes the dimension of $X$, $\mathrm{dim}^{*}_{\forall}(X)$ denotes the generic dimension of $X$, and $\mathrm{dim}^{*}_{\exists}(X)$ denotes the locally generic dimension of $X$, then we always have that 
$\mathrm{dim}(X)\geq \mathrm{dim}^{*}_{\forall}(X)\geq \mathrm{dim}^{*}_{\exists}(X)$.
While the inequalities are strict in general, the last two quantities are equal whenever $X$ is generically ergodic.



The obstruction, developed in \cite{GOP}, for classifying orbit equivalence relations by CLI group actions relied on the fact that the image of a combinatorial  $1$-cube under an injective digraph homomorphism can never be a self-loop.  Working with  digraph embeddings---rather than just injective homorphisms---allows us to utilize combinatorial $n$-cubes as obstructions for classification under $*$-reductions, even when $n>1$.

\begin{theorem}\label{T: main 2}
Let  $X$ be a Polish $G$-space and let $Y$ be a Polish $H$-space where $G$ and $H$ are Polish groups. Let also $n\in \{0,1,2,\ldots\}\cup\{\infty\}$.
\begin{enumerate}
\item If $X$ is generically $n$-dimensional and the dimension of $Y$ is less than $n$, then $E^G_X$ does not $*$-reduce to $E^H_Y$.
\item If the dimension of $X$ is less than $n$ and $Y$ is locally generically $n$-dimensional, then $E^G_X$ does not $*$-reduce to $E^H_Y$. 
\end{enumerate}
\end{theorem}
\begin{proof}
The proof is a direct consequence of Theorem \ref{T: main}.
\end{proof}

In the rest of this section we develop examples of Polish $G$-spaces $X$ which are $n$-dimensional and locally generically $n$-dimensional.  We then use Theorem \ref{T: main 2} to deduce that these spaces are pairwise incomparable under $*$-reductions.

Let $\mathcal{L}$ be a countable language. We consider the space  $\mathcal{X}_{\mathcal{L}}$  of $\mathcal{L}$-structures with domain $\omega$. For any formula $\varphi(x)$ and any tuple $a$ from $\omega$, define \[[\varphi(a)] = \{A\in \mathcal{X}_{\mathcal{L}} \mid A\models \varphi(a)\}.\] Then a subbasis for the topology on $\mathcal{X}_{\mathcal{L}}$ is given by the sets of the form $[R(a)]$ and $[\lnot R(a)]$ for every $n$-ary relation symbol $R$ in $\mathcal{L}$ and every $n$-tuple $a$ from $\omega$, and $[f(a) = b]$ for every $n$-ary function symbol $f$ in $\mathcal{L}$, every $n$-tuple $a$ from $\omega$, and every element $b\in \omega$. That is, $\mathcal{X}_{\mathcal{L}}$ is a Polish space homeomorphic to a product of Cantor spaces $2^{(\omega^{n})}$ for each $n$-ary relation symbol in $\mathcal{L}$ and Baire spaces $\omega^{(\omega^n)}$ for each $n$-ary function symbol in $\mathcal{L}$. 

Consider now the Polish group $S_\infty$, of all bijections from $\omega$ to $\omega$, endowed with the pointwise convergence topology. There is a natural continuous action of $S_\infty$ on $\mathcal{X}_{\mathcal{L}}$; namely, if $g\in S_\infty$ and $A\in\mathcal{X}_{\mathcal{L}}$, then $gA$ is the unique $B\in\mathcal{X}_{\mathcal{L}}$ so that for every tuple  $a=(a_0,\ldots,a_{n-1})$ in $\omega$ and every quantifier free formula $\varphi$ we have that 
\[B\models \varphi(a_0,\ldots,a_{n-1}) \iff A\models \varphi(g^{-1}(a_0),\ldots,g^{-1}(a_{n-1}))\]
In other words, $g A = B$ if and only if $g$ is an isomorphism $A\to B$. The orbit equivalence relation on $\mathcal{X}_{\mathcal{L}}$ induced by the logic action is denoted $\simeq_{\mathrm{iso}}$. 

Recall that a sequence $(g_n)$ in any Polish group is left-Cauchy if it is Cauchy with respect to some left-invariant metric, and moreover, that this is equivalent with $(g_n)$ being Cauchy with respect to every left-invariant metric; see \cite{Becker}. In $S_{\infty}$ a compatible left-invariant metric  is given by $d_l(g,h)=1/2^m$,  where $m$ is the least natural number with $g(m)\neq h(m)$. Elements $\gamma$ of the left completion of $S_{\infty}$ can be identified with injections  $\gamma\colon \omega\to\omega$ which are not necessarily surjective; see e.g. \cite{Gaopaper}. 
Since for every left-invariant metric $d_l$, in any Polish group, the metric $d_r(g,h):=d_l(g^{-1},h^{-1})$ is right-invariant, we have that a sequence $(g_n)$ in $S_{\infty}$ is right-Cauchy if and only if $(g_n^{-1})$ is left-Cauchy. Similarly to the logic action, where $g A =B$ if and only if  $g$ is an isomorphism from $A$ to $B$, the following proposition states that right-embeddings from $A$ to $B$ correspond to  model-theoretic embeddings from $A$ to $B$. For left-embeddings the situation is a bit more complicated; see \cite{Becker}.

\begin{proposition}\label{Proposition Emb}
Let $A,B\in \mathcal{X}_{\mathcal{L}}$ and let $(g_n)$ be a right-Cauchy sequence in $S_{\infty}$. Let also $\gamma\colon \omega\to \omega$ be the injective map that is the limit of $(g^{-1}_n)$. The following are equivalent:
\begin{enumerate}
\item $(g_n B)$ converges to $A$;
\item $\gamma$ is an embedding from $A$ to  $B$. 
\end{enumerate}
\end{proposition}
\begin{proof}
Let $\varphi$ be a quantifier free formula and let $a$ be a tuple in $\omega$. By definition of the logic action,  $\varphi(a)$
holds in $g_n B$ for large $n$ if and only if $\varphi(g^{-1}_n a)$ holds in $B$ for large $n$.
Since $(g^{-1}_n)$ is converging to $\gamma$, the later is equivalent to $B \models \varphi(\gamma a)$.
The rest follows from the fact that $g_nB$ converges to $A$ if and only if for all $a$ and $\varphi$ as above and for all large enough $n$ we have that $A\models \varphi(a)\iff g_n B \models \varphi(a)$.
\end{proof}

When $\mathcal{K}$ is a class of $\mathcal{L}$-structures, we write $\mathrm{Mod}_{\omega}(\mathcal{K})$ for the subspace of $\mathcal{X}_{\mathcal{L}}$ consisting of structures in $\mathcal{K}$. Abusing notation, we also define \[[\varphi(a)] = \{A\in \mathrm{Mod}_{\omega}(\mathcal{K})\mid A\models \varphi(a)\}.\] 
When $\mathrm{Mod}_{\omega}(\mathcal{K})$ is an invariant subspace of $\mathcal{X}_{\mathcal{L}}$, the logic action descends to $\mathrm{Mod}_{\omega}(\mathcal{K})$. We are particularly interested in the case when $\mathrm{Mod}_{\omega}(\mathcal{K})$ is a $G_\delta$, and hence Polish, subspace of $\mathcal{X}_{\mathcal{L}}$. Note that $\mathrm{Mod}_{\omega}(\mathcal{K})$ is an invariant Borel subspace of $\mathcal{X}_{\mathcal{L}}$ whenever $\mathcal{K}$ is axiomatizable by a countable $\mathcal{L}_{\omega_1,\omega}$-theory, and it is often possible to check that $\mathrm{Mod}_{\omega}(\mathcal{K})$ is $G_\delta$ by looking at the form of this axiomatization.

Given a countable language $\mathcal{L}$, we denote by $\widehat{\mathcal{L}}$ the language $\mathcal{L}\cup \{P_i\mid i\in \omega\}$, where the $P_i$ are unary relation symbols which do not appear in $\mathcal{L}$. Given a class $\mathcal{K}$ of $\mathcal{L}$-structures, we denote by $\widehat{\mathcal{K}}$ the class of $\widehat{\mathcal{L}}$-structures whose reduct to $\mathcal{L}$ is in $\mathcal{K}$ and such that no two elements satisfy exactly the same set of predicates $P_i$, i.e. the satisfy the following axiom:
\[\forall x\forall y\, \left((x\neq y)\rightarrow \bigvee_{i\in \omega} (P_i(x)\not\leftrightarrow P_i(y))\right).\tag{A1}\] We call the structures in $\widehat{\mathcal{K}}$ \textbf{labeled} $\mathcal{K}$-structures.  

\begin{example}\label{Example:ctbl sets of reals}
If $\mathcal{L}_\emptyset$ is the empty language and $\mathcal{S}$ is the class of all sets then every structure in $\mathrm{Mod}_{\omega}(\widehat{\mathcal{S}})$ is essentially a sequence of distinct reals (elements of $2^\omega$). Up to isomorphism, such a structure is essentially a countable set of reals. The orbit equivalence relation $\simeq_{\mathrm{iso}}$ on $\mathrm{Mod}_{\omega}(\widehat{\mathcal{S}})$ is often denoted $=^{+}$. 

For any class of $\mathcal{L}$-structures $\mathcal{K}$, $\mathrm{Mod}_{\omega}(\widehat{\mathcal{K}})$ and $\mathrm{Mod}_{\omega}(\mathcal{K})\times \mathrm{Mod}_{\omega}(\widehat{\mathcal{S}})$ are isomorphic as Polish $S_\infty$-spaces. Note that $\mathrm{Mod}_{\omega}(\widehat{\mathcal{S}})$ is an invariant $G_\delta$ subspace of $\mathcal{X}_{\widehat{\mathcal{L}_{\emptyset}}}$. 
Hence, when $\mathrm{Mod}_{\omega}(\mathcal{K})$ is an invariant $G_\delta$ subspace of $\mathcal{X}_{\mathcal{L}}$, then $\mathrm{Mod}_{\omega}(\widehat{\mathcal{K}})$ is also an invariant $G_\delta$ subspace of $\mathcal{X}_{\widehat{\mathcal{L}}}= \mathcal{X}_{\mathcal{L}}\times \mathcal{X}_{\widehat{\mathcal{L}_\emptyset}}$.
\end{example}

Besides the fact that they give rise to natural equivalence relations which generalize $=^{+}$ it is convenient to consider labeled $\mathcal{K}$-structures for two more reasons. First, the logic action on $\mathrm{Mod}_{\omega}(\widehat{\mathcal{K}})$ automatically has meager orbits. This is crucial if one wants to apply Theorem \ref{T: main 2} on $\mathrm{Mod}_{\omega}(\widehat{\mathcal{K}})$ in any meaningful way since the existence of a comeager orbit in any $G$-space $X$ implies that $X$ is generically $0$-dimensional. Second, if $A$ and $B$ are labeled $\mathcal{K}$-structures, then there is at most one embedding $A\to B$. As a consequence, every diagram of labeled $\mathcal{K}$-structures and embeddings is automatically commutative. This labeling trick will allows us to work at the level of the embeddability relation (i.e., the information contained within the Becker digraph) without having to keep track of the composition relation between embeddings. We leave it to future work to develop a functorial version of the results in this paper, perhaps within the framework of Polish groupoids \cite{Lupini}. 

The main examples of logic actions that we will consider below consist of \emph{labeled BKL$_n$-structures}. Fix $n\geq 1$ and let $\mathcal{L}_n$ be the language which contains a collection $\{s_i\mid i\in \omega \}$ of $n$-ary function symbols, and a collection $\{R_j \mid j\in \omega\}$ of $n$-ary relation symbols. A {\bf BKL$_n$-structure} is any $\mathcal{L}_n$-structure $A$ which satisfies the following list of axioms:
\begin{enumerate}[(B1)]
\item The $n$-ary relations $(R_j)_{j\in \omega}$ partition $A^n$.
\item For any $n$-tuple $a = (a_0,\dots,a_{n-1})$ satisfying $R_j$, we have $s_i(a) = a_0$ for all functions $s_i$ with $i>j$. 
\item There is no substructure-independent set of size $(n+1)$: If $|B| = n+1$, then there is some $b\in B$ such that $b$ is in the substructure generated by $B\setminus \{b\}$. 
\item $A$ is locally finite: For any finite tuple $a$ in $A$, the substructure generated by $a$ is finite. 
\end{enumerate}

We will denote by $\mathcal{B}_n$ and by $\mathcal{B}_n^{\mathrm{fin}}$ the classes of all BKL$_n$-structures and all finite BKL$_n$-structures, respectively. Note that the empty structure is a BKL$_n$-structure.  

The BKL$_n$ structures were introduced by Baldwin, Koerwein, and Laskowski in~\cite{BKL}, based on a similar construction by Laskowski and Shelah in~\cite{LS}. In~\cite{BKL}, the classes $\mathcal{B}_n$ and $\mathcal{B}_n^{\mathrm{fin}}$ were called $\hat{\mathbf{K}}^{n-1}$ and $\mathbf{K}_0^{n-1}$, respectively. 

Our goal for the rest of this section is to prove the following theorem.
 
\begin{theorem}\label{thm:bklmain}
For every $n\geq 1$ the $S_{\infty}$-space $\mathrm{Mod}_\omega(\widehat{\mathcal{B}_{n}})$ is $(n-1)$-dimensional and locally generically $(n-1)$-dimensional.
\end{theorem}

We will conclude with the proof of Theorem \ref{thm:bklmain} at the end of the section after we collect the necessary lemmas. First we record the following corollary which is an immediate consequence of Theorem \ref{T: main 2} and Theorem \ref{thm:bklmain}.

\begin{corollary}\label{cor:main}
There is no $*$-reduction from $(\mathrm{Mod}_\omega(\widehat{\mathcal{B}_{m}}),\simeq_{\mathrm{iso}})$ to $(\mathrm{Mod}_\omega(\widehat{\mathcal{B}_{n}}),\simeq_{\mathrm{iso}})$, when $m\neq n$.
\end{corollary}

Recall from the beginning of this section the poset category $\Delta^{n-1}=(\mathcal{P}(\{0,\ldots,n-1\}),\subseteq)$ where $[n]=\{0,1,\ldots,n-1\}$ is the terminal object. Consider also the full subcategory $\partial\Delta^{n-1}$ of $\Delta^{n-1}$ whose set of objects is $\Delta^{n-1}\setminus \{[n]\}$. We view each class of structures $\mathcal{K}$ as a category whose arrows are embeddings. An {\bf $n$-cube in $\mathcal{K}$} is a functor $\mathcal{A}$ from $\Delta^{n-1}$ to $\mathcal{K}$, i.e., a pair $\mathcal{A} = \big((A_{\sigma})_{\sigma},(f^{\sigma}_{\tau})_{\sigma\subseteq\tau} \big)$, where $(A_{\sigma})_{\sigma}$ is a collection of structures from $\mathcal{K}$, indexed by elements $\sigma$ of $\Delta^{n-1}$; together with embeddings $f^{\sigma}_{\tau}\colon  A_{\sigma}\to A_{\tau}$, so that each $f^\sigma_\sigma$ is the identity map, and for every $\sigma\subseteq\tau\subseteq\rho$ we have that $f^{\tau}_{\rho} \circ f^{\sigma}_{\tau}=f^{\sigma}_{\rho}$. Similarly, a \textbf{partial $n$-cube} in $\mathcal{K}$ is a functor from $\partial\Delta^{n-1}$ to $\mathcal{K}$. 

An $n$-cube in $\mathcal{K}$ is \textbf{disjoint} if $f^\sigma_{\sigma\cup \tau}(A_\sigma) \cap f^\tau_{\sigma\cup \tau}(A_\tau) = f^{\sigma\cap \tau}_{\sigma\cup\tau}(A_{\sigma\cap\tau})$ for all $\sigma$ and $\tau$. Similarly, a partial $n$-cube is disjoint if the same condition holds whenever $\sigma\cup \tau\subsetneq [n]$. The class $\mathcal{K}$ has \textbf{disjoint $n$-amalgamation} if every disjoint partial $n$-cube can be extended to a disjoint $n$-cube.
Let $\mathcal{A} = \big((A_\sigma),(f^\sigma_\tau)\big)$ and $\mathcal{B} = \big((B_\sigma),(g^\sigma_\tau)\big)$ be disjoint $n$-cubes. A \textbf{disjoint embedding} $\mathcal{A}\to \mathcal{B}$ is a family of embeddings $h_\sigma\colon A_\sigma\to B_\sigma$ such that for all $\sigma\subseteq \tau$, $h_\tau\circ f^{\sigma}_\tau = g^{\sigma}_\tau \circ h_\sigma$, and for all $\sigma$ and $\tau$, \[(h_{\sigma\cup \tau}\circ f^\sigma_{\sigma\cup \tau})(A_\sigma) \cap g^\tau_{\sigma\cup \tau}(B_{\tau}) = (h_{\sigma\cup \tau}\circ f^{\sigma\cap \tau}_{\sigma\cup \tau})(A_{\sigma\cap \tau}).\]
An $n$-cube is {\bf reducible} if there are $\sigma,\tau\in\Delta^{n-1}$ with $\sigma\not\subseteq\tau$ and some embedding  $f\colon A_{\sigma}\to A_{\tau}$ so that $f^{\tau}_{\rho}\circ f=f^{\sigma}_{\rho}$ and $f\circ f^{\rho'}_{\sigma}=f^{\rho'}_{\tau}$, for all $\rho'\subseteq \sigma\cap \tau$ and $\sigma\cup\tau \subseteq \rho$. Equivalently, an $n$-cube is reducible if and only if there are $\sigma,\tau\in \Delta^{n-1}$ with $\sigma\not\subseteq\tau$ such that the image of $A_{\sigma} $ in $A_{[n]}$ is contained in the image of $A_{\tau}$ in $A_{[n]}$. 

\begin{lemma}\label{lem:bklprops}
\begin{enumerate}
\item The class $\mathcal{B}_n^{\mathrm{fin}}$ has disjoint $k$-amalgamation for all $1\leq k\leq n$.
\item If $\mathcal{A}$ is an $n$-cube in $\mathcal{B}_n$ such that $A_\emptyset$ is infinite, then $\mathcal{A}$ is reducible.
\end{enumerate}
\end{lemma}
\begin{proof}
\cite[Theorem 3.1.1]{BKL} essentially proves (1) but we repeat the short proof for completeness. Suppose $((A_\sigma)_{\sigma},(f^\sigma_\tau)_{\sigma\subseteq \tau})$ is a disjoint partial $k$-cube in $\mathcal{B}_n^{\mathrm{fin}}$. Let $A_{[k]}$ be the colimit of this diagram in the category of sets. Explicitly, $A_{[k]}$ is the disjoint union of the $A_\sigma$ for $\sigma\in \partial \Delta^{k-1}$, with $f^{\sigma}_\tau(a)$ and $f^{\sigma}_{\tau'}(a)$ identified for all $\sigma\subseteq \tau,\tau'$ and $a\in A_\sigma$. For all $\sigma\in \partial \Delta^{k-1}$, we define $f^\sigma_{[k]}$ to be the canonical inclusion $A_\sigma\to A_{[k]}$. Then the disjointness conditions are satisfied. 

It remains to make $A_{[k]}$ into a structure in $\mathcal{B}_n^{\mathrm{fin}}$ by defining the functions $s_i$ and relations $R_j$ on all $n$-tuples. So let $a = (a_0,\dots,a_{n-1})$ be an $n$-tuple in $A_{[k]}$. If $a$ is in the image of any $A_\sigma$ for $\sigma\in \partial \Delta^{k-1}$, there is a unique way to define the $s_i$ and $R_j$ so that $f^\sigma_{[k]}$ is an embedding. So we may assume that $a$ is not contained in the image of any $A_\sigma$. Enumerate $A_{[k]}$ as $c_0,\dots,c_N$, set $R_N(a)$, and define $s_i(a) = c_i$ for all $i\leq N$ and $s_i(a) = a_0$ for all $i>N$. Axioms (B1) and (B2) are satisfied by construction, and (B4) is trivially satisfied, since $A_{[k]}$ is finite. For (B3), suppose for contradiction that $B\subseteq A_{[k]}$ is substructure-independent, with $|B| = n+1$. Let $B = \{b_0,\dots,b_n\}$, and for all $0\leq i\leq n$, let $b^i = (b_0,\dots,b_{i-1},b_{i+1},\dots,b_n)$. If each $b^i$ is contained in the image of some $A_{[k]\setminus \{j\}}$, then since $k\leq n$, there are $i\neq i'$ and $j$ such that $b^i$ and $b^{i'}$ are both contained in the image of $A_{[k]\setminus \{j\}}$, and hence all of $B$ is contained in the image of $A_{[k]\setminus \{j\}}$. This contradicts the fact that $A_{[k]\setminus \{j\}}\in \mathcal{B}_n^{\mathrm{fin}}$. So there is some $i$ such that $b^i$ is an $n$-tuple not contained the image of any $A_\sigma$. Then $b_i = c_j$ for some $0\leq j\leq N$, and $s_j(b^i) = b_i$, so $B$ is not substructure-independent. 
 
For (2), assume towards contradiction that there is an irreducible $n$-cube $\mathcal{A}$ in $\mathcal{B}_n$ such that each structure $A_\sigma$ in $\mathcal{A}$ is infinite. Since $\mathcal{A}$ is not reducible, for each $0\leq i< n$, the image of $A_{\{i\}}$ in $A_{[n]}$ is not contained in the image of $A_{[n]\setminus \{i\}}$ in $A_{[n]}$. Pick $x_i$ in $f^{\{i\}}_{[n]}(A_{\{i\}})\setminus f^{[n]\setminus \{i\}}_{[n]}(A_{[n]\setminus \{i\}})$ for all $i$. Now by axiom (B4), $\langle x_0,\dots,x_{n-1}\rangle$ is finite, so we can pick some $x_n\in f^\emptyset_{[n]}(A_\emptyset)\setminus \langle x_0,\dots,x_{n-1}\rangle$.

Then the set $\{x_0,\dots,x_{n}\}$ is substructure-independent. We have already seen that $x_n\notin \langle x_0,\dots,x_{n-1}\rangle$. And for all $0\leq i < n$, $\langle x_0,\dots,x_{i-1},x_{i+1},\dots,x_n\rangle \subseteq f^{[n]\setminus \{i\}}
_{[n]}(A_{[n]\setminus \{i\}})$, so $x_i\notin \langle x_0,\dots,x_{i-1},x_{i+1},\dots,x_n\rangle$. This contradicts (B3).  
\end{proof}

The statement of Lemma~\ref{lem:bklprops}(2) also holds for $\widehat{\mathcal{B}_n}$, by taking reducts to $\mathcal{L}_n$.
As a consequence, there are no irreducible $n$-cubes in $\widehat{\mathcal{B}_n}$ all of whose structures lie in $\mathrm{Mod}_\omega(\widehat{\mathcal{B}_{n}})$.  Most of our remaining work is to prove that, in contrast, $\mathrm{Mod}_\omega(\widehat{\mathcal{B}_{n}})$ has many irreducible $k$-cubes whenever $k<n$. We begin by observing that $\mathcal{B}^{\mathrm{fin}}_n$ is a \Fraisse class when $n\geq 2$ (this was also used in~\cite{BKL}). Then we will find an irreducible $k$ cube in $\mathcal{B}_n$ for $1\leq k<n$ such that every structure in the cube is isomorphic to the \Fraisse limit of $\mathcal{B}^{\mathrm{fin}}_n$. Finally, we will use the existence of this $k$-cube and a Baire category argument to find $k$-cubes in any comeager subset of $\mathrm{Mod}_\omega(\widehat{\mathcal{B}_{n}})$.

\begin{lemma}\label{lem:bklspaces}
\begin{enumerate}
\item Let $A$ be a finite BKL$_n$-structure with domain $\{a_0,\dots,a_k\}$. Then there is a finite conjunction of atomic and negated atomic $\mathcal{L}_n$-sentences with parameters from $A$, denoted $\theta_A(a_0,\dots,a_{k})$, such that for any BKL$_n$-structure $B$, $B\models \theta_A(b_0,\dots,b_{k})$ if and only if the map $a_i\mapsto b_i$ is an embedding $A\to B$.
\item When $n\geq 2$, $\mathcal{B}_n^{\mathrm{fin}}$ is a \Fraisse class.
\end{enumerate}
\end{lemma}
\begin{proof}

For any $n$-tuple $c$ from $A$, let $j_c$ be the unique natural number such that $A\models R_{j_c}(c)$, and let $d_c^i = s_i(c)$ for all $i\in \omega$. Then let $\theta_A(a_0,\dots,a_k)$ be the conjunction of the following formulas:
\begin{itemize}
\item $a_i\neq a_j$, for all $1\leq i < j \leq k$. 
\item $R_{j_c}(c)$, for each $n$-tuple $c$ from $A$.
\item $s_i(c) = d^i_c$, for each $n$-tuple $c$ from $A$ and each $i\leq j_c$. 
\end{itemize}

This finite piece of the diagram of $A$ is sufficient to determine the rest of its diagram. Indeed, by axioms (B1) and (B2), $R_{j_c}(c)$ implies $\lnot R_{j}(c)$ for all $j\neq j_c$ and $s_i(c) = c_0$ for all $i>j_c$.  This establishes (1). 

For (2), it is clear from the definition that $\mathcal{B}_n^{\mathrm{fin}}$ has the hereditary property, and the amalgamation property is Lemma~\ref{lem:bklprops}(1) in the case $k=2$. The amalgamation property implies the joint embedding property, since $\mathcal{B}_n^{\mathrm{fin}}$ includes the empty structure. Finally, the fact that $\mathcal{B}_n^{\mathrm{fin}}$ is countable up to isomorphism follows from part (1), since there are only countably many $\mathcal{L}_n$ formulas. 
\end{proof}

We denote by $\mathcal{B}^*_{n}$ the class of all structures in $\mathcal{B}_{n}$ which are isomorphic to the \Fraisse limit of $\mathcal{B}_{n}^{\mathrm{fin}}$. These are exactly the countably infinite BKL$_n$-structures which satisfy the following \textbf{extension axiom} for every pair of  finite BKL$_n$-structures $A\subseteq B$ with domains $\{a_0,\dots,a_k\}$ and $\{a_0,\dots,a_k,a_{k+1},\dots, a_\ell\}$, respectively: \[\forall x_0,\dots,x_k\, (\theta_A(x_0,\dots,x_k)\rightarrow \exists x_{k+1},\dots,x_\ell \,\theta_B(x_0,\dots,x_\ell)).\] 
 
\begin{lemma}\label{L: in-out} Suppose $1\leq k<n$. Let $\mathcal{A} = \big((A_{\sigma}),(f^{\sigma}_{\tau}) \big)$ be a disjoint $k$-cube in $\mathcal{B}_{n}^{\mathrm{fin}}$. For every $\rho\in \Delta^{k-1}$ and every embedding $h\colon A_{\rho} \to B$ in $\mathcal{B}_{n}^{\mathrm{fin}}$, there is a disjoint $k$-cube $\mathcal{B} = \big((B_{\sigma}),(g^{\sigma}_{\tau}) \big)$ in $\mathcal{B}_{n}^{\mathrm{fin}}$ and a disjoint embedding $(h_\sigma)\colon \mathcal{A}\to \mathcal{B}$ such that $B_\rho = B$ and $h_\rho = h$, and further, for all $\tau$ such that $\rho\not\subseteq \tau$, we have $B_\tau = A_\tau$, $h_\tau\colon A_\tau\to B_\tau$ is the identity map, and $g^\sigma_\tau = f^\sigma_\tau$ for all $\sigma\subseteq \tau$.
\end{lemma}
\begin{proof}
We define $\mathcal{B}$ and $(h_{\sigma})$ in stages, ensuring that the parts of $\mathcal{B}$ and $(h_\sigma)$ that we have defined so far are functorial and satisfy the relevant disjointness conditions. Namely, for each $\tau$, we must check:
\begin{enumerate}[(a)]
\item For all $\sigma\subseteq \sigma'\subseteq \tau$, $g^\tau_\tau$ is the identity, and $g^{\sigma'}_\tau\circ g^\sigma_{\sigma'} = g^\sigma_\tau$. 
\item For all $\sigma\subseteq \tau$, $h_\tau\circ f^\sigma_\tau = g^\sigma_\tau\circ h_\sigma$. 
\item For all $\sigma\cup \sigma' = \tau$, $g^\sigma_{\tau}(B_\sigma) \cap g^{\sigma'}_{\tau}(B_{\sigma'}) = g^{\sigma\cap \sigma'}_{\tau}(B_{\sigma\cap \sigma'})$.  
\item For all $\sigma\cup \sigma' = \tau$, $(h_{\tau}\circ f^\sigma_{\tau})(A_\sigma) \cap g^{\sigma'}_{\tau}(B_{\sigma'}) = (h_{\tau}\circ f^{\sigma\cap \sigma'}_{\tau})(A_{\sigma\cap \sigma'})$.  
\end{enumerate}

For $\tau$ such that $\rho\not\subseteq\tau$, set $B_{\tau}=A_{\tau}$, let $h_{\tau}\colon A_{\tau}\to B_{\tau}$ be the identity map, and let $g^{\sigma}_{\tau}=f^{\sigma}_{\tau}$ for $\sigma\subseteq \tau$, as required in the statement of the lemma. Conditions (a), (b), (c), and (d) are satisfied for all such $\tau$; (a) by functoriality of $\mathcal{A}$, (b) trivially, and (c) and (d) because $\mathcal{A}$ is disjoint. For $\tau$ with $\rho\subseteq\tau$, we proceed by induction on the size $m=| \tau \setminus \rho|$. 

For $m=0$, i.e., for $\tau=\rho$, let $h_{\tau}\colon A_{\tau}\to B_{\tau}$ be $h\colon A_{\rho}\to B$. Let $g^\tau_\tau$ be the identity and $g^{\sigma}_{\tau}= h \circ f^{\sigma}_{\tau}$ for $\sigma\subsetneq \tau$. Condition (a) is trivially satisfied when $\sigma' = \tau$, and letting $\sigma\subseteq \sigma'\subsetneq \tau$, we have $\rho\not\subseteq \sigma'$, so $g^{\sigma'}_\tau\circ g^\sigma_{\sigma'} = h\circ f^{\sigma'}_\tau \circ f^{\sigma}_{\sigma'} = h\circ f^{\sigma}_\tau = g^\sigma_\tau$. Condition (b) is trivially satisfied when $\sigma = \tau$, so it suffices to consider $\sigma\subsetneq \tau$. Then $\rho\not\subseteq\sigma$, and we have $g^{\sigma}_{\tau}\circ h_{\sigma}=h_{\tau} \circ  f^{\sigma}_{\tau}$, since $h_\sigma$ is the identity map. 

Condition (c) is trivial when $\sigma = \tau$ or $\sigma' = \tau$, so we may assume $\rho\not\subseteq\sigma$ and $\rho\not\subseteq \sigma'$. Then   \[g^\sigma_\tau(B_\sigma) \cap g^{\sigma'}_\tau(B_{\sigma'}) = h(f^\sigma_\tau(A_\sigma)\cap f^{\sigma'}_\tau(A_{\sigma'})) = h(f^{\sigma\cap \sigma'}_\tau(A_{\sigma\cap \sigma'})) = g^{\sigma\cap \sigma'}_\tau(B_{\sigma\cap \sigma'}).\] 

Similarly, condition (d) is trivial when $\sigma' = \tau$, so we may assume $\rho\not\subseteq \sigma'$. Then  \[(h_{\tau}\circ f^\sigma_{\tau})(A_\sigma) \cap g^{\sigma'}_{\tau}(B_{\sigma'}) = h(f^\sigma_\tau(A_\sigma)\cap f^{\sigma'}_\tau(A_{\sigma'})) = (h_{\tau}\circ f^{\sigma\cap \sigma'}_{\tau})(A_{\sigma\cap \sigma'}).\]

Assume now that we have defined $B_{\tau}$, $g^\sigma_\tau$, and $h_{\tau}$ for all $\tau$ with $|\tau \setminus \rho|<m$ and all $\sigma\subseteq \tau$, such that conditions (a), (b), (c), and (d) are satisfied. We consider $\tau$ with $|\tau \setminus \rho|=m$. Let $\{a_0,\ldots,a_{m-1}\}$ be an enumeration of the set  $\tau \setminus \rho$. 

First we define a disjoint partial $(m+1)$-cube, i.e.\ a functor $F\colon \partial\Delta^{m}\to \mathcal{B}^{\mathrm{fin}}_n$. 
Let
\[
F(\sigma) = \begin{cases} 
B_{\rho\cup \{a_i\mid i\in \sigma,i\neq m\}} & m\in \sigma\\
A_{\rho\cup \{a_i\mid i\in \sigma\}} & m\notin \sigma.
\end{cases}
\]
The image of the relation $\sigma\subseteq \sigma'$  under $F$ is defined as the unique choice amongst $f^{\sigma}_{\sigma'}$, $g^{\sigma}_{\sigma'}$, or $h_{\sigma'} \circ f^{\sigma}_{\sigma'}$ which makes sense. Functoriality of $F$ follows from conditions (a) and (b) and functorality of $\mathcal{A}$, and disjointness of the $(m+1)$-cube follows from conditions (c) and (d) and disjointness of $\mathcal{A}$. 

Since $m+1\leq k+1\leq n$, $\mathcal{B}^{\mathrm{fin}}_n$ has disjoint $(m+1)$-amalgamation by Lemma~\ref{lem:bklprops}, so we can extend $F$ to a disjoint $(m+1)$-cube $G\colon \Delta^m\to \mathcal{B}^{\mathrm{fin}}_n$. Set $B_{\tau}=G([m+1])$, and let $h_{\tau}$ be the image of the relation $[m]\subseteq[m+1]$ under $G$. For any subset $\sigma$ of $\tau$, we define $g^{\sigma}_{\tau}$ to be the appropriate embedding in the image of $G$, if $\rho\subseteq \sigma$, or $h_{\tau}\circ f^{\sigma}_{\tau}$ if not. Now conditions (a) and (b) follow from functoriality of $G$ and induction, and conditions (c) and (d) follow from disjointness of $G$ and induction. 
\end{proof}

\begin{lemma}\label{L: in-out2}
For all $1\leq k < n$, there is an irreducible $k$-cube in $\mathcal{B}^*_n$. 
\end{lemma}
\begin{proof} 
We build a sequence of disjoint $k$-cubes in $\mathcal{B}_n^{\mathrm{fin}}$, $\mathcal{A}^i = \big((A^i_\sigma),(f^{\sigma,i}_\tau)\big)$ for $i\in \omega$, together with disjoint embeddings $(h^i_\sigma) \colon \mathcal{A}^i\to \mathcal{A}^{i+1}$ for all $i$. Then for each $\sigma$, we will let $C_\sigma$ be the directed colimit of the $A^i_\sigma$, and for each $\sigma\subseteq \tau$, let $g^\sigma_\tau$ be the natural map $C_\sigma\to C_\tau$ induced by the maps $f^{\sigma,i}_\tau$ for all $i$. This defines a $k$-cube $\mathcal{C} = \big((C_\sigma),(g^\sigma_\tau)\big)$. The idea is to carry out the \Fraisse construction simultaneously for each $\sigma$, so that each $C_\sigma$ in $\mathcal{B}_n^*$. 

We begin with the empty $k$-cube $\mathcal{A}^0$, where $A_\sigma^0$ is the empty structure for all $\sigma$. Suppose we have defined $\mathcal{A}^i$, and we are given some $\rho$, some substructure $B\subseteq A^i_\rho$, and some embedding $e\colon B\to B'$ in $\mathcal{B}^{\mathrm{fin}}_n$. By disjoint $2$-amalgamation, there is a structure $A_\rho^{i+1}\in \mathcal{B}_n^{\mathrm{fin}}$ and embeddings $h^i_\rho\colon A_\rho^i\to A_\rho^{i+1}$ and $e'\colon B'\to A_\rho^{i+1}$ such that $h^i_\rho\res B = e'\circ e$ and $h^i_\rho(A_\rho)\cap e'(B) = h^{i}_{\rho}(B)$. By Lemma~\ref{L: in-out}, we can extend $A_\sigma^{i+1}$ to a disjoint $k$-cube $\mathcal{A}^{i+1}$ and extend $h^i_\rho$ to a disjoint embedding $(h^i_\sigma)\colon \mathcal{A}^i\to \mathcal{A}^{i+1}$. Taking care in the course of the construction to handle every embedding from a substructure of some $A^i_\rho$ in this way, we may ensure that each $C_\sigma$ is in $\mathcal{B}_n^*$.

It remains to show that the $k$-cube $\mathcal{C}$ is irreducible. Suppose $\sigma\not\subseteq \tau$. We will show that $g^\sigma_{[k]}(C_\sigma)\not\subseteq g^\tau_{[k]}(C_\tau)$. Pick some $i$ such that the extension $h_\sigma^i\colon A_\sigma^i\to A_\sigma^{i+1}$ is determined by some proper extension of a substructure of $A_\sigma^i$. Since $A_\sigma^{i+1}$ is defined by disjoint $2$-amalgamation, there is some $x_{i+1}\in  A_\sigma^{i+1}\setminus h_\sigma^i(A_\sigma^i)$, and further, by Lemma~\ref{L: in-out}, we know that $h^{i}_\tau\colon A_\tau^{i}\to A_\tau^{i+1}$ is the identity map, since $\sigma\not\subseteq \tau$. 

Let $y_{i+1} = f^{i+1,\sigma}_{[k]}(x_{i+1})\in A^{i+1}_{[k]}$. Then $y_{i+1}\notin f^{i+1,\tau}_{[k]}(A^{i+1}_\tau)$. Indeed, if it were, then there would be some $z\in A^i_\tau$ such that $f^{i+1,\tau}_{[k]}(h^i_\tau(z)) = y_{i+1}$, since $h^i_\tau$ is the identity map. But then since $(h^i_\sigma)$ is a disjoint embedding, $x_{i+1}$ and $z$ would both be in the image of some element in $A^i_{\sigma\cap \tau}$, and in particular $x_{i+1}$ would be in the image of some element of $A^i_\sigma$, contradicting the choice of $x_{i+1}$. 

Now for $j>i+1$, define $x_{j+1} = h^j_\sigma(x_j)\in A^{j+1}_\sigma$ and $y_{j+1} = h^j_{[k]}(y_j)\in A^{j+1}_{[k]}$ by induction. Let $x$ and $y$ be the common images of the $x_j$ and $y_j$ in $C_\sigma$ and $C_{[k]}$. Then $y = g^\sigma_{[k]}(x)$. We will show by induction on $j\geq i+1$ that $y_j\notin f^{j,\tau}_{[k]}(A^j_\tau)$, from which it follows that $y\in g^\sigma_{[k]}(C_\sigma)\setminus  g^\tau_{[k]}(C_\tau)$. 

We have already established the base case, when $j = i+1$. For the inductive step, if $y_{j+1} = h^j_{[k]}(y_j) \in f^{j+1,\tau}_{[k]}(A^{j+1}_\tau)$, then since $(h_\sigma)$ is a disjoint embedding, then $y_j\in f^{j,\tau}_{[k]}(A^j_\tau)$, contradicting the inductive hypothesis. 
\end{proof}

It will be convenient to encode an entire $k$-cube in a single structure. Fix $1\leq k<n$, and consider the language \[\mathcal{L}^k_{n}=\mathcal{L}_{n}\cup\{D_{\sigma}\mid \sigma\in \Delta^{k-1}\},\] where 
the $D_{\sigma}$ are unary relation symbols. Let $\mathcal{D}^k_{n}$ be the class of all $\mathcal{L}^k_{n}$-structures $M$ such that:
\begin{enumerate}
\item $D^M_{[k]} = M$.
\item Each $D^M_{\sigma}$ is an $\mathcal{L}_n$-substructure of $M\res \mathcal{L}_n$.
\item $D^M_\sigma\subseteq D^M_\tau$ when $\sigma\subseteq \tau$.
\item $(D^M_{\sigma})_\sigma$ together with the inclusion maps is an irreducible $k$-cube in $\mathcal{B}^*_{n}$.
\end{enumerate}

\begin{lemma}
\begin{enumerate}
\item $\mathrm{Mod}_\omega(\mathcal{B}_{n})$ is an invariant $G_\delta$ subspace of $\mathcal{X}_{\mathcal{L}_n}$.
\item A basis for the topology on $\mathrm{Mod}_\omega(\mathcal{B}_{n})$ is given by the collection of sets $[\theta_A(a_0,\dots,a_{k})]$, where $A$ is a finite BKL$_n$-structure with domain $\{a_0,\dots,a_{k}\}\subseteq \omega$. 
\item $\mathrm{Mod}_\omega(\mathcal{B}^{*}_{n})$ is an invariant \emph{dense} $G_{\delta}$ subspace of $\mathrm{Mod}_\omega(\mathcal{B}_{n})$.
\item $\mathrm{Mod}_\omega(\widehat{\mathcal{B}^{*}_{n}})$ is an invariant \emph{dense} $G_{\delta}$ subspace of $\mathrm{Mod}_\omega(\widehat{\mathcal{B}_{n}})$.
\item For $1\leq k<n$, $\mathrm{Mod}_\omega(\mathcal{D}^k_{n})$ is an invariant \emph{nonempty} $G_\delta$ subspace of $\mathcal{X}_{\mathcal{L}^k_{n}}$.
\item For $1\leq k<n$, $\mathrm{Mod}_\omega(\widehat{\mathcal{D}^k_{n}})$ is an invariant \emph{nonempty} $G_\delta$ subspace of $\mathcal{X}_{\widehat{\mathcal{L}^k_{n}}}$.
\end{enumerate}
\end{lemma}

\begin{proof}
For (1), we give an axiomatization of $\mathcal{B}_n$ in $\mathcal{L}_{\omega_1,\omega}$:
\begin{itemize}
\item $\forall x_0,\dots,x_{n-1} \bigvee_{j\geq 0} R_j(x_0,\dots,x_{n-1})$;
\item $\forall x_0,\dots,x_{n-1} \bigwedge_{0\leq i<j} \lnot (R_i(x_0,\dots,x_{n-1})\land R_j(x_0,\dots, x_{n-1})$;
\item $\forall x_0,\ldots,x_{n-1} \bigwedge_{0\leq i< j} (R_i(x_0,\dots,x_{n-1})\rightarrow s_j(x_0,\ldots,x_{n-1})=x_0)$;
\item $\forall x_0,\ldots,x_{n} \bigvee_{0\leq k\leq n} \bigvee_{t(x_1,\dots,x_{k-1},x_{k+1},\dots,x_n)} t(x_1,\dots,x_{k-1},x_{k+1},\dots,x_n) = x_k$,\\ 
where $t$ ranges over all terms in the variables $x_0,\dots,x_n$, with $x_k$ omitted.
\item $\forall x_0,\dots,x_{k}(\bigwedge_{1\leq i<j\leq k} x_i\neq x_j \rightarrow \bigvee_{A = \{a_0,\dots,a_\ell\}} \exists x_{k+1},\dots,x_\ell\, \theta_A(x_0,\dots,x_\ell))$,\\
where $A$ ranges over all isomorphism types of structures in $\mathcal{B}_n^{\mathrm{fin}}$ with a fixed enumeration $a_0,\dots,a_\ell$, where $k\leq \ell$. Note that this is only a countable disjunction, since each such $A$ corresponds to the finitary $\mathcal{L}_n$-formula $\theta_A$. 
\end{itemize}
From the form of the axiomatization it follows that $\mathrm{Mod}_\omega(\mathcal{B}_{n})$ is $G_\delta$ in $\mathcal{X}_{\mathcal{L}_n}$.

For (2), let $U$ be a nonempty open subset of $\mathrm{Mod}_\omega(\mathcal{B}_{n})$. We may assume that $U = [\varphi(a)]$, where $\varphi(x)$ is a finite conjunction of atomic and negated atomic $\mathcal{L}_n$-formulas. Let $A\in [\varphi(a)]$. By axiom (B4), $A$ is locally finite. Let $B$ be the substructure of $A$ generated by the tuple $a$, and fix an enumeration of $B$ by a tuple $b$. Then $A\in [\theta_B(b)] \subseteq [\varphi(a)]$. 

For (3), the fact that $\mathrm{Mod}_\omega(\mathcal{B}^{*}_{n})$ is an invariant $G_{\delta}$ subspace of $\mathrm{Mod}_\omega(\mathcal{B}_{n})$ is clear from the axiomatization by extension axioms. Density follows from (2), since every finite BKL$_n$-structure with domain a subset of $\omega$ is a substructure of a structure in $\mathrm{Mod}_\omega(\mathcal{B}^{*}_{n})$. Then (4) follows immediately. 

For (5), note that conditions (1)-(3) in the definition of $\mathcal{D}^k_n$ are closed, the condition that the $k$-cube is irreducible is open, and the condition that each $D^M_\sigma$ is in $\mathcal{B}^*_n$ is $G_\delta$, by relativizing the extension axioms for $\mathcal{B}^*_n$ to $D_\sigma$. To see that it is nonempty, let $((A_\sigma),(f^\sigma_\tau))$ be the irreducible $k$-cube in $\mathcal{B}^*_n$ constructed in Lemma~\ref{L: in-out2}. We may assume that $M = A_{[k]}$ has domain $\omega$ and define the relations $D^M_\sigma$ as $f^\sigma_{[k]}(A_\sigma)$. Then (6) follows immediately.
\end{proof}

If $M\in \mathrm{Mod}_\omega(\widehat{\mathcal{D}^k_{n}})$, then for all $\sigma\in \Delta^{k-1}$, we view $D^M_\sigma$ as a $\widehat{\mathcal{L}_n}$-substructure of $M\res \widehat{\mathcal{L}_n}$. Notice that the domain of $D^M_{\sigma}$ is always an infinite subset of $\omega$, since $M\res \widehat{\mathcal{L}_n}$ is in $\mathrm{Mod}_\omega(\widehat{\mathcal{B}^{*}_{n}})$. Let $e^M_\sigma\colon \omega \to \omega$ enumerate $D^M_\sigma$ in increasing order. Then there is a unique structure $N\in \mathrm{Mod}_\omega(\widehat{\mathcal{B}^{*}_{n}})$ such that $e^M_\sigma\colon N\to M$ defines an $\widehat{L_n}$-isomorphism between $N$ and  $D^M_\sigma$. The assignment $M\mapsto N$ defines a map:
\[\phi_{\sigma}\colon\mathrm{Mod}_\omega(\widehat{\mathcal{D}^k_{n}})\to\mathrm{Mod}_\omega(\widehat{\mathcal{B}^{*}_{n}}).\]

\begin{lemma}\label{L: cont open}
For all $1\leq k<n$, and for every $\sigma\in \Delta^{k-1}$, the map $\phi_{\sigma}\colon\mathrm{Mod}_\omega(\widehat{\mathcal{D}^k_{n}})\to\mathrm{Mod}_\omega(\widehat{\mathcal{B}^{*}_{n}})$ is continuous and open.
\end{lemma}

\begin{proof}
For continuity, let $U$ be an open set in $\mathrm{Mod}_\omega(\widehat{\mathcal{B}^{*}_{n}})$. We may assume $U$ has the form $[\psi(a)]$, where $\psi$ is a finite conjunction of atomic and negated atomic $\widehat{\mathcal{L}_n}$-formulas and $a = (a_1,\dots,a_t)$ is a tuple from $\omega$ with $a_1<\dots<a_t$. Let $M\in \phi_\sigma^{-1}(U)$. Then $M\models \psi(e^M_\sigma(a))$. Let $\chi$ be the conjunction of all formulas of the form $D_\sigma(b)$ and $\lnot D_\sigma(b)$ which are satisfied in $M$, for $b\leq e^M_\sigma(a_t)$, and let $V = [\psi(e^M_\sigma(a))\land \chi]$. Then $V$ is an open set in $\mathrm{Mod}_\omega(\widehat{\mathcal{D}^k_{n}})$ and $M\in V\subseteq \phi_\sigma^{-1}(U)$. Indeed, if $M'\in V$, then $e^{M'}_\sigma$ agrees with $e^{M}_\sigma$ on the initial segment of $\omega$ up to $a_t$, so $M'\models \psi(e^{M'}_\sigma(a))$, and $\phi_\sigma(M')\in U$.

For openness, let $U$ be an open set in $\mathrm{Mod}_\omega(\widehat{\mathcal{D}^k_{n}})$. We may assume $U$ has the form $[\psi(a)]$, where $\psi$ is a finite conjunction of atomic and negated atomic $\widehat{\mathcal{L}^k_n}$-formulas and $a = (a_1,\dots,a_t)$ is a tuple from $\omega$ with $a_1<\dots<a_t$. Note also that $\psi$ mentions only finitely many of the labeling predicates, $P_{i_1},\dots,P_{i_\ell}$. Let $N\in \phi_\sigma(U)$, and let $M\in U$ with $\phi_\sigma(M) = N$. Let $b\in \omega$ be the largest element such that $e^{M}_\sigma(b)$ is in the tuple $a$ in $M$, and let $B = \langle 0,\dots,b\rangle$ be the substructure of $N$ generated by the initial segment of $\omega$ up to $b$. Enumerate $B$ as $\{b_1,\dots,b_m\}$. 

Recall that $\theta_B(b_1,\dots,b_m)$ determines the  $\mathcal{L}_n$-isomorphism type of $B$. Let $\chi$ be the conjunction of all formulas of the form $P_{i_j}(b_{j'})$ and $\lnot P_{i_j}(b_{j'})$ which are satisfied in $N$, for $1\leq j\leq \ell$ and $1\leq j' \leq m$. And let $V = [\theta_B(b_1,\dots,b_m)\land \chi]$. Then $V$ is an open set in $\mathrm{Mod}_\omega(\widehat{\mathcal{B}^{*}_{n}})$ and $N\in V\subseteq \phi_\sigma(U)$. 

To see this, given $N'\in V$, we will find $M'\in U$ such that $\phi_\sigma(M') = N'$. We obtain $M'$ by modifying our original structure $M$ in stages. 
\begin{enumerate}
\item Since $N'\models \theta_B(b_1,\dots,b_m)$, $M\models \theta_B(e^M_\sigma(b_1),\dots,e^M_\sigma(b_m))$. So $e^M_\sigma$ restricts to an $\mathcal{L}_n$-embedding $B\to D^M_\sigma$. Since $D^M_\sigma\res \mathcal{L}_n$ and $N'\res\mathcal{L}_n$ are both isomorphic to the \Fraisse limit of $\mathcal{B}^{\text{fin}}_n$, this embedding extends to an $\mathcal{L}_n$-isomorphism $f\colon N'\to D^M_\sigma$ which agrees with $e^M_\sigma$ on $B$. We will use this $f$ to re-enumerate $M$ so that $\phi_\sigma$ produces a structure which agrees with the target structure $N'$ in the reduct to the language $\mathcal{L}_n$.

\item We view the composition $e^M_\sigma \circ f^{-1}\colon D^M_\sigma\to D^M_\sigma$ as a permutation of $D^M_\sigma\subseteq \omega$. Extending this permutation by the identity outside $D^M_\sigma$ defines a permutation $g\in S_\infty$. Let $M_0 = g(M)$ (by the logic action of $S_\infty$ on $\mathrm{Mod}_\omega(\widehat{\mathcal{D}^k_{n}})$. Note that $g$ is the identity on the tuple $a$. Indeed, if $a_i\notin D^M_\sigma$, then $g(a_i) = a_i$ by definition. And if $a_i\in D^M_\sigma$, then there is some $b'$ such that $e^M_\sigma(b') = a_i$. Since $a_i\leq a_t$, $b'\leq b$, so $b'\in B$. But since $f$ agrees with $e^M_\sigma$ on $B$, $g(a_i) = e^M_\sigma(f^{-1}(a_i)) = a_i$. It follows that $M_0\in [\psi(a)] = U$. Further, since $g$ fixes $D^M_\sigma$ set-wise, we have $D^{M_0}_\sigma = D^M_\sigma$ as subsets of $\omega$, and $e^{M_0}_\sigma = e^M_\sigma$. But $g^{-1}\circ e^{M_0}_\sigma = f$ is an $\mathcal{L}_n$-isomorphism $N'\to D^M_\sigma$, so $e^{M_0}_\sigma$ is an $\mathcal{L}_n$-isomorphism $N'\to D^{M_0}_\sigma$. It follows that $\phi_\sigma(M_0)\res \mathcal{L}_n = N'\res\mathcal{L}_n$. 

\item It remains to improve this to equality of $\widehat{\mathcal{L}_n}$-structures by relabeling $M_0$ by the predicates $P_i$. Let $M_1 = M_0\res\mathcal{L}^k_n$, dropping all the labels. We now define the $\widehat{\mathcal{L}_n}$-structure $M'$ expanding $M_0$. Define the labels on the elements of $D^{M_0}_\sigma$ so that $e^{M_0}_\sigma$ is an $\widehat{\mathcal{L}_n}$-isomorphism. Note that no two elements get exactly the same labels, since this is true in $N'$. Now assign labels arbitrarily to the remaining elements of $M_1$, only making sure that no two elements get exactly the same labels and that for each element $a_i$ of the tuple $a$, the new label of $a_i$ agrees with its old label on the predicates $P_{i_1},\dots,P_{i_\ell}$. In fact, this is already the case when $a_i$ is in $D^{M_0}_\sigma$, since $N'\in V$ and the information about the predicates $P_{i_1},\dots,P_{i_\ell}$ is included in $\chi$. We have now arranged that $\phi_\sigma(M') = N'$. 

\item It remains to check that $M'\in U = [\psi(a)]$. This follows from the fact that $M_0\in [\psi(a)]$, and $M'$ agrees with $M_0$ on all atomic formulas in $\mathcal{L}^k_n$, as well as the values of all the predicates mentioned in $\psi$ on the tuple $a$. \qedhere
\end{enumerate}
\end{proof}

We now have all the ingredients necessary to prove Theorem~\ref{thm:bklmain}.

\begin{proof}[Proof of Theorem~\ref{thm:bklmain}]
First we observe that having an irreducible $k$-cube in $\widehat{\mathcal{B}_{n}}$ with structures lying in some invariant subset $C$ of $\mathrm{Mod}_\omega(\widehat{\mathcal{B}_{n}})$ is the same as having an embedding of the combinatorial $k$-cube in the Becker graphs $\mathcal{B}(C/S_{\infty})$. This follows from Proposition \ref{Proposition Emb} and the fact that we are working with labeled structures; see the discussion after the definition of labeled $\mathcal{K}$-structures.
By Lemma~\ref{lem:bklprops}(2) it follows that the dimension of $\mathrm{Mod}_\omega(\widehat{\mathcal{B}_{n}})$ is at most $(n-1)$.

When $n=1$, it is immediate that $\mathrm{Mod}_\omega(\widehat{\mathcal{B}_{n}})$ is locally generically $(n-1)$ dimensional (this just means that every non-meager invariant set with the Baire property is nonempty). So in the remainder of the proof, we assume $n>1$. 

First, we show that $\mathrm{Mod}_\omega(\widehat{\mathcal{B}_{n}})$ is generically $(n-1)$-dimensional. Let $C\subseteq \mathrm{Mod}_\omega(\widehat{\mathcal{B}_{n}})$ be a comeager set. Since $\mathrm{Mod}_\omega(\widehat{\mathcal{B}^{*}_{n}})$ is dense $G_\delta$ in $\mathrm{Mod}_\omega(\widehat{\mathcal{B}_{n}})$, the restriction $C^* = C\cap \mathrm{Mod}_\omega(\widehat{\mathcal{B}^{*}_{n}})$ is comeager in $\mathrm{Mod}_\omega(\widehat{\mathcal{B}_{n}})$. For each $\sigma\in \Delta^{k-1}$, let $C^*_\sigma = \phi_\sigma^{-1}(C^*)\subseteq \mathrm{Mod}_\omega(\widehat{\mathcal{D}^k_{n}})$. Continuous open maps are category preserving, so each $C^*_\sigma$ is comeager in $\mathrm{Mod}_\omega(\widehat{\mathcal{D}^k_{n}})$. Since this space is nonempty and Polish, $\bigcap_{\sigma\in \Delta^{k-1}} C^*_\sigma$ is nonempty. Let $M$ be a structure in the intersection. 
We define a $k$-cube $\mathcal{A}$ by setting $A_\sigma = \phi_\sigma(M)$. In particular, $A_{[k]} = M\res \widehat{\mathcal{L}_n}$. To define the connecting maps, let $f^\sigma_{[k]} = e^M_\sigma \colon A_\sigma\to A_{[k]}$. Now for any $\sigma\subseteq \tau\subseteq [k]$, the map $f^\sigma_{[k]}$ factors through the map $f^\tau_{[k]}$ by a unique map $f^\sigma_\tau$. By construction, $\mathcal{A}$ is irreducible and $A_\sigma\in C^*\subseteq C$. 

Finally, we claim that $\mathrm{Mod}_\omega(\widehat{\mathcal{B}_{n}})$ is generically ergodic when $n>1$, so that the fact that it is generically $(n-1)$-dimensional immediately implies that it is locally generically $(n-1)$-dimensional. By~\cite[Proposition 6.1.9]{Gao}, it suffices to show that there is a structure in $\mathrm{Mod}_\omega(\widehat{\mathcal{B}_{n}})$ with a dense orbit.

Consider the $\widehat{\mathcal{L}_n}$-theory which axiomatizes labeled BKL$_n$-structures (axioms (A1), (B1), (B2), (B3), and (B4)) and contains the additional axiom:  
\[\bigwedge_{A=\{a_0,\ldots,a_k\}} \bigwedge_{(T_0,F_0),\ldots,(T_k, F_k)} \exists x_0,\ldots,x_k  \bigg(\theta_A(x_0,\ldots,x_k) \wedge \bigwedge_{i\leq k} \big(\bigwedge_{j\in T_i} P_j(x_i) \land \bigwedge_{j\in F_i} \lnot P_j(x_i)\big)\bigg),\]
where $A = \{a_0,\dots,a_k\}$ varies over all isomorphism types of structures in $\mathcal{B}^{\mathrm{fin}}_{n}$, and for each $i$, $(T_i,F_i)$ is a pair of disjoint finite subsets of $\omega$.  It is easy to  see that there are structures in $\mathrm{Mod}_\omega(\widehat{\mathcal{B}_{n}})$ which satisfy this theory and that every such structure has a dense orbit under the $S_{\infty}$ action.
\end{proof}

\begin{remark}\label{Remark fin}
Keeping track of the properties of $\mathcal{B}_n$ used in the proof of Theorem~\ref{thm:bklmain}, we see that for any class of structures $\mathcal{K}$ satisfying the following hypotheses, the $S_{\infty}$-space $\mathrm{Mod}_\omega(\widehat{\mathcal{K}})$ is locally generically $(\geq (n-1))$-dimensional: 
\begin{enumerate}
\item $\mathcal{K}$ is closed under substructure. 
\item $\mathrm{Mod}_\omega(\mathcal{K})$ is a $G_\delta$-subspace of $\mathcal{X}_{\mathcal{L}}$. 
\item $\mathcal{K}^{\mathrm{fin}}$ is a \Fraisse class, whose \Fraisse limit is in $\mathcal{K}$. 
\item $\mathcal{K}^{\mathrm{fin}}$ has disjoint $k$-amalgamation for all $1\leq k\leq n$. 
\item $\mathcal{K}^{\mathrm{fin}}$ is \emph{separable}: this is the condition in Lemma~\ref{lem:bklspaces}(1). 
\end{enumerate}
It follows, for example, that if $\mathcal{K}$ is the class $\mathcal{S}$ of all sets or $\mathcal{K}$ is the class  of all graphs, then $\mathrm{Mod}_\omega(\widehat{\mathcal{K}})$ is locally generically $\infty$-dimensional. In both cases Theorem \ref{T: main 2} implies that $\simeq_{\mathrm{iso}}$ on $\mathrm{Mod}_\omega(\widehat{\mathcal{K}})$ is incomparable to $\simeq_{\mathrm{iso}}$ on $\mathrm{Mod}_\omega(\widehat{\mathcal{B}_{n}})$, for all $n>0$. In particular, $=^{+}$ is incomparable to  $(\mathrm{Mod}_\omega(\widehat{\mathcal{B}_{n}}),\simeq_{\mathrm{iso}})$, for all $n>0$; see Example \ref{Example:ctbl sets of reals}.
\end{remark}

\bibliography{bibliography}
\bibliographystyle{alpha}

\end{document}